\documentclass[12pt]{amsart}
\usepackage{amsmath,amscd,amsthm,amsfonts, amssymb,amsxtra}
\usepackage{epsf}
\usepackage[all]{xy}
\usepackage{hyperref}
  \hypersetup{colorlinks=true,citecolor=blue}
  \usepackage{graphicx}
\CDat

\SelectTips{cm}{}
\newtheorem{thm}{Theorem}[section]  
\newtheorem*{un-no-thm}{Theorem}
\newtheorem{cor}[thm]{Corollary}     
\newtheorem{lem}[thm]{Lemma}         
\newtheorem{prop}[thm]{Proposition}  
        
\newtheorem{add}[thm]{Addendum}
\newtheorem{bigthm}{Theorem}

\newtheorem*{A.1}{Conjecture A.1}
\newtheorem*{B.1}{Conjecture B.1}
\newtheorem*{conjEpd}{Conjecture C}
\newtheorem*{C.1}{Conjecture C.1}
\newtheorem*{conjEFpd}{Conjecture D}
\newtheorem*{D.1}{Theorem D.1}

\theoremstyle{definition}
\newtheorem{defn}[thm]{Definition}   

\theoremstyle{definition}

\theoremstyle{definition}

\theoremstyle{remark}
\newtheorem{rem}[thm]{Remark}

\newtheorem*{acks}{Acknowledgements}

\newtheorem*{ass}{Assumption}

\newtheorem{notation}[thm]{Notation}

\DeclareMathOperator\Diff{Diff}

\DeclareMathOperator\holim{holim}

\DeclareMathOperator\hocodim{hocodim}
\DeclareMathOperator\Aut{Aut}
\DeclareMathOperator\fib{fib}
\DeclareMathOperator\rel{rel}

\begin{document}
\title[Multiple disjunction]{Multiple disjunction for spaces \\ of smooth embeddings}
\date{\today}
\author{Thomas G. Goodwillie}
\address{Brown University, Providence, RI 02912}
\email{tomg@math.brown.edu}
\author{John R. Klein}
\address{Wayne State University, Detroit, MI 48202}
\email{klein@math.wayne.edu}
\begin{abstract}  We obtain multirelative
connectivity statements about spaces of smooth embeddings, deducing these from a similar result
about spaces of Poincar\'e embeddings that was established in \cite{Goodwillie-Klein} and a similar result about condordance embeddings that was established in \cite{thesis}. 
\end{abstract}
\thanks{Both authors have been partially supported by the NSF}
\maketitle
\setlength{\parindent}{15pt}
\setlength{\parskip}{1pt plus 0pt minus 1pt}
\def\Sp{\bold S\bold p}
\def\vo{\varOmega}
\def\vs{\varSigma}
\def\smsh{\wedge}
\def\flush{\flushpar}
\def\id{\text{id}}
\def\dbslash{/\!\! /}
\def\:{\colon}
\def\Bbb{\mathbb}
\def\bold{\mathbf}
\def\cal{\mathcal}
\def\CE{\mathit{CE}}
\def\EF{\mathit{EF}}
\def\Wh{\mathit{Wh}}

\setcounter{tocdepth}{1}
\tableofcontents
\addcontentsline{file}{sec_unit}{entry}

\section{Introduction}

This paper forms a pair with \cite{Goodwillie-Klein}, and to some extent the introduction to that paper serves as an introduction to this one, too.

Our results are multirelative connectivity statements: they assert that certain cubical diagrams of spaces are \lq highly connected\rq\ in the sense of being $k$-cartesian for a given value of $k$. For terminology and basic facts about connectivity and cubical diagrams, including the \lq higher Blakers-Massey Theorem\rq , see the early sections of \cite{Goodwillie_CALC2} or the appendix of \cite{Goodwillie-Klein}.

Our main results are Theorems \ref{E} through \ref{EF} below. We regard Theorems \ref{E}, \ref{symm}, \ref{excision}, and \ref{int} as one result looked at in four different ways. Theorem \ref{EF} is closely related.

Let $E(P,N)$ be the space of all smooth embeddings of a compact manifold $P$ in the manifold $N$. It is elementary to show that when $Q$ is a submanifold of $N$ then the inclusion map $E(P,N-Q)\to E(P,N)$ is $(n{-}p{-}q{-}1)$-connected, where $n$, $p$, and $q$ are the dimensions. This is true by simple dimension-counting (transversality): a $k$-parameter family of maps $P\to N$ is generically disjoint from $Q$ if $k{+}p<n{-}q$. 

Theorem \ref{E} is a multirelative generalization of this fact.
Briefly, the statement is the following: Let $N$ and $P$ be as above and suppose that $Q_1,\dots ,Q_r$ is a collection of pairwise disjoint submanifolds of $N$. For $S\subset \underline r=\lbrace 1,\dots ,r\rbrace$ write $Q_S=\cup_{i\in S}Q_i$. Then the $r$-dimensional cubical diagram $E(P,N-Q_\bullet)$ formed by the spaces $E(P,N-Q_S)$ is $(1{-}p{+}\Sigma_{i=1}^r(n{-}q_i{-}2))$-cartesian. We succeed in proving this in all cases except the one corresponding to ordinary knot theory, when $n{=}3$, $p{=}1$, and $q_i{=}1$. Most of the work goes into dealing with the case when the codimensions $n{-}p$ and $n{-}q_i$ are at least $3$. The proof uses techniques from homotopy theory, surgery, and concordance (pseudoisotopy) theory. 

\begin{rem}\label{F}Let $F(P,N)$ be the space of all maps from $P$ to $N$. The cube $F(P,N-Q_\bullet)$ is always $(1{-}p+\Sigma_{i=1}^r(n{-}q_i{-}2))$-cartesian. This follows from the fact that the cube $N-Q_\bullet$ is, by the higher Blakers-Massey theorem \cite[2.4]{Goodwillie_CALC2},  $(1{+}\Sigma_{i=1}^r(n{-}q_i{-}2))$-cartesian.
\end{rem}
\begin{rem}\label{vw}In an Appendix to \cite{Goodwillie-Klein} a statement similar to Theorem \ref{E} but with a generally much lower number is proved using only dimension-counting and the higher Blakers-Massey theorem. It says that $E(P,N-Q_\bullet)$ is $(1{-}rp{+}\Sigma_{i=1}^r(n{-}q_i{-}2))$-cartesian, with no restriction on dimensions. This will be useful in \S\ref{summing-up} and \S\ref{cod2} for handling some low-dimensional cases.
\end{rem}

Theorem \ref{symm} is a variant of Theorem \ref{E}, easily seen to be equivalent to it. In Theorem \ref{symm} the role of $P$ is no different from that of any $Q_i$.

Theorem \ref{excision}, a more elaborate version of Theorem \ref{symm}, is the statement that guarantees strong convergence of Taylor towers of embedding functors in codimension three or more in Weiss's manifold functor calculus (see \cite[th.\ 1.4, ex.\ 2.2, th.\ 2.3]{GW}). 
There the goal is to systematically compare $E(M,N)$ with spaces $E(U,N)$ where $U\subset M$ is small, or special -- for example, to describe $E(M,N)$ as a homotopy limit of spaces $E(U,N)$ where each $U$ is the union of finitely many disjoint disks. Theorem \ref{excision} expresses a connectivity property (\lq analyticity\rq\ or \lq approximate higher excision\rq\ ) of the functor $U\mapsto E(U,N)$.

Theorem \ref{int} is Theorem \ref{symm} restated in terms of moduli spaces of manifolds rather than spaces of embeddings.

Theorem \ref{EF} is a multirelative generalization of the fact that the inclusion $E(P,N)\to F(P,N)$ of the space of all embeddings into the space of all maps is $(n{-}2p{-}1)$-connected. It is closely related to Theorem \ref{E}, and their proofs are inextricably mixed together.

We now state the results in more detail and in a little more generality, and explain how they are related to each other.

\subsection{Conventions}

When speaking of embeddings of a compact manifold $P$ in a manifold $N$, 
we allow $P$ to have a boundary, all or part of which may be embedded in the boundary of $N$. The part that is in the boundary of $N$ never moves. Thus $P$ will be a manifold triad: its boundary is the union of two parts, $\partial_0P$ and $\partial_1P$, intersecting at a corner $\partial\partial_0P=\partial\partial_1P$. (Any of these sets might be empty or disconnected.) The convention is that some embedding $\partial_0P\to \partial N$ is fixed in advance and $E(P,N)$ denotes the space of all embeddings $P\to N$ restricting to this one. 

We also give ourselves the flexibility of working with statements that refer not to the dimension of a submanifold but rather to its {\it handle dimension}, essentially the dimension of a spine. 

\begin{defn}  A compact smooth manifold triad $(P;\partial_0 P,\partial_1 P)$ has {\it handle dimension} $\le p$ (relative to $\partial_0 P$)
if $P$ can be built up from a collar
$\partial_0 P\times I$ by attaching handles of index at most $p$.
\end{defn}
Of course handle dimension is less than or equal to dimension. 
\begin{rem}\label{thicken}Handle dimension is preserved when $P$ is replaced by a disk bundle over $P$:
Suppose that $(P;\partial_0P,\partial_1P)$ is a compact $p$-dimensional manifold triad and $P$ is the base of
a vector bundle $\xi$ with inner product. Then the total space $D(\xi)$ of the unit disk bundle has handle dimension $\le p$ if $\partial_0 D(\xi)$ is taken to be
the part of $D(\xi)$ lying over $\partial_0P$.\end{rem}

If $P$ is a submanifold of an $n$-manifold $N$ and $p$ is its handle dimension (relative to $P\cap\partial N$) then $n-p$ is called its handle codimension.

\begin{rem}If $P$ has handle dimension $\le p$, then it also has homotopy spine dimension $\le p$ in the sense of \cite{Goodwillie-Klein}. That is, the pair $(P,\partial_0P)$ is homotopy equivalent to a cellular pair of relative dimension at most $p$ and the pair $(P,\partial_1P)$ is $(n{-}p{-}1)$-connected.\end{rem}

\subsection{The First Main Result}\label{ThmA}

\begin{bigthm}\label{E} Let $N$ be a smooth $n$-manifold. Let $r\ge 1$ and suppose that $Q_1,\dots ,Q_r$ are compact smooth manifold triads with handle dimensions $q_i$, and that they are embedded disjointly in $N$ with $\partial_0Q_i=Q_i\cap\partial N$. Write $Q_S$ for the disjoint union $\cup_{i\in S}Q_i$. Let $P$ be a compact manifold triad of handle dimension $p$, with $\partial_0P$ embedded in $\partial N$ disjointly from the $Q_i$. Then the $r$-cube $E(P,N-Q_\bullet)$ is $(1{-}p +\sum_{i=1}^r (n{-}q_i{-}2))$-cartesian, except possibly in the case when $n{=}3$, $p{=}1$, and $q_i{=}1$ for all $i$.
\end{bigthm}

Note that the same statement with \lq dimension\rq\ instead of \lq handle dimension\rq\ is included as a special case. In fact, the general case could be deduced from this special case. However, in order to avoid normal bundle issues we prefer to work with the opposite extreme: the special case when $P$ and $Q_i$ have dimension $n$. 
Let us show that the theorem follows from this codimension zero case.

To see that it follows from the special case in which $P$ is $n$-dimensional, we replace $P$ by a tubular neighborhood. The cube $E(P,N-Q_\bullet)$ will be $k$-cartesian if for every point $e\in E(P,N)$ the cube made up of the homotopy fibers (over $e$) of the maps $E(P,N-Q_S)\to E(P,N)$ is $k$-cartesian. Let $\xi$ be a vector bundle over $P$ (a potential normal bundle for an embedding $e:P\to N$) whose restriction to $\partial_0P$ is identified with the normal bundle of $\partial_0P$ in $\partial N$. Make the disk bundle $D(\xi)$ into a manifold triad as in Remark \ref{thicken}, having the same handle dimension $p$ as $P$.
We have a fibration 
$$E(D(\xi),N-Q_S)\to E(P,N-Q_S)$$
(restriction to zero section) for each $S$. Its fiber over a given embedding is homotopy equivalent to the space of isomorphisms, fixed on $\partial_0P$, between $\xi$ and the normal bundle. In particular this fiber is independent of $S$, and this is the key point of the proof: it implies that the homotopy fiber of $E(D(\xi),N-Q_S)\to E(D(\xi),N)$ is equivalent to that of $E(P,N-Q_S)\to E(P,N)$. Thus the cube $E(P,N-Q_\bullet)$ must be $k$-cartesian if for every possible $\xi$ the cube $E(D(\xi),N-Q_\bullet)$ is $k$-cartesian for the same $k$. 

We can go further, reducing to the case in which $P$ is a single handle (or the result of attaching a single handle to a collar on $\partial_0P$). Induct on the number of handles in a handle decomposition of $P$. Suppose that $P=H\cup A$ where $H$ is a handle of index at most $p$ and the conclusion holds when $P$ is replaced by $A$. There is a fibration 
$$
E(P,N-Q_S)\to E(A,N-Q_S),
$$
the restriction map.
By assumption the $r$-cube $E(A,N-Q_\bullet)$ is $(1{-}p +\sum_{i=1}^r (n{-}q_i{-}2))$-cartesian. The $r$-cube $E(P,N-Q_\bullet)$ will be $(1{-}p +\sum_{i=1}^r (n{-}q_i{-}2))$-cartesian if the $(r+1)$-cube 
$$
E(P,N-Q_\bullet)\to E(A,N-Q_\bullet)
$$ 
is $(1{-}p +\sum_{i=1}^r (n{-}q_i{-}2))$-cartesian. For this it suffices if for each point $e\in E(A,N-Q_{\underline r})$ the $r$-cube of fibers
$$
E(H,N'-Q_\bullet)
$$
is $(1{-}p +\sum_{i=1}^r (n{-}q_i{-}2))$-cartesian, where $N'$ is the closed complement of $e(A)$ in $N$.

For completeness we now also explain why Theorem \ref{E} follows from the special case in which $p$ is the dimension of $P$. In fact, the case when $P=H$ is a $p$-handle follows from the case in which $(P;\partial_0P,\partial_1 P)\cong (D^p;S^{p-1},\emptyset)$, by the reverse of an argument given above. View $H$ as a tubular neighborhood of a copy of $D^p$ in $N$, therefore a disk bundle $D(\xi)$ over $D^p$, and use again that the fiber of the restriction map $E(D(\xi),N-Q_S)\to E(D^p,N-Q_S)$ is independent of $S$ up to homotopy.

We omit the even more elementary reduction to the case when each $Q_i$ has dimension $n$ (and thence to the case when $Q_i$ is a $q_i$-handle or a $q_i$-disk).

\subsection{A Symmetrical Formulation of the First Main Result}

The following is equivalent to Theorem \ref{E}. More precisely, Theorem \ref{symm} for $r$ is equivalent to Theorem \ref{E} for $r{-}1$.

\begin{bigthm}\label{symm}Let $N$ be a smooth $n$-manifold. Let $r\ge 2$ and suppose that $Q_1,\dots ,Q_r$ are compact smooth manifold triads, and that the manifolds $\partial _0Q_i$, are embedded pairwise disjointly in $\partial N$. Let $q_i$ be the handle dimension of $Q_i$ with respect to $\partial_0Q_i$. Write $Q_S$ for the disjoint union $\cup_{i\in S}Q_i$. Then the $r$-cube $E(Q_\bullet,N)$ is $(3{-}n{+}\Sigma_{i=1}^r(n{-}q_i{-}2))$-cartesian, except possibly in the case when $n=3$ and $q_i=1$ for all $i$.
\end{bigthm}

Again this result includes as a special case the same statement with handle dimension replaced by dimension. Again the general statement follows easily from the codimension zero case. 

Let us show that Theorems \ref{E} and \ref{symm} are equivalent. We may assume codimension zero. Let $(N;Q_1,\dots ,Q_r)$ be as in Theorem \ref{symm} and single out $Q_r$ for special treatment. ($Q_r$ will be $P$.) For each $S\subset \underline {n{-}1}$ we have the fibration
$$
E(Q_S\cup Q_r,N)\to E(Q_S,N).
$$
Its fiber over a given point $e\in E(Q_S,N)$ is $E(Q_r,N-e(Q_S))$. View the $r$-cube mentioned in Theorem \ref{symm} as a map of $(r{-}1)$-cubes
$$
E(Q_\bullet\cup Q_r,N)\to E(Q_\bullet,N),
$$
where $\bullet$ now runs through subsets of $\underline {n{-}1}$. Whenever an embedding $e:Q_{\underline{r{-}1}}\to N$ is chosen, then $E(Q_\bullet,N)$ becomes an $(r{-}1)$-cube of based spaces, and the fibers over the base points form an $(r{-}1)$-cube isomorphic to $E(Q_r,N-e(Q_\bullet))$. 

The original $r$-cube is $k$-cartesian if and only if for each such choice of $e$ the $(r{-}1)$-cube of fibers is $k$-cartesian. 
Write $P=Q_r$ and $p=q_r$. Take $k$ to be
$$
3{-}n{+}\Sigma_{i=1}^r(n{-}q_i{-}2)=1{-}p{+}\Sigma_{i=1}^{r-1}(n{-}q_i{-}2).
$$
Thus the conclusion of Theorem \ref{symm} holds for 
\mbox{$(N,Q_1,\dots ,Q_{r-1},Q_r=P)$} if and only if, for every way of embedding the disjoint union of \mbox{$Q_1,\dots,Q_{r-1}$} in $N$, the conclusion of Theorem \ref{E} holds.

\subsection{Excision/Analyticity Formulation of the First Main Result}

The next result is also equivalent to Theorem \ref{symm}.
Again let $N$ be an $n$-manifold. Let $M$ be a compact $m$-manifold triad with $\partial_0M$ embedded in $\partial N$. Assume that $M$ contains compact $m$-manifold triads $Q_1,\dots ,Q_r$, disjoint from one another and from $\partial_0M$, with $\partial_1Q_i=Q_i\cap \partial_1M$. Let $q_i$ be the handle dimension of $Q_i$ relative to $\partial_0Q_i$. (For example, $Q_i$ might be a single handle, a tubular neighborhood of an $(m-q_i)$-disk in $M$ that is disjoint from $\partial_0M$ and transverse to $\partial_1M$.)
Let $Q_S$ be the union of the $Q_i$ for $i\in S$ and consider the $r$-dimensional cubical diagram $E(M-Q_\bullet ,N)$. 

\begin{bigthm} \label{excision}  Let $N$, $M$, $Q_i$, and $q_i$ be as above, with $r\ge 2$.
Then
the $r$-cubical diagram $E(M-Q_\bullet,N)$ is $(3{-}n{+}\Sigma_i(n{-}q_i{-}2))$-cartesian, except possibly in the case when $n=3$ and $q_1=\dots =q_r=1$.
\end{bigthm}

Note that in particular the cube is $(3{-}n{+}r(n{-}m{-}2))$-cartesian if $n{-}m\ge 3$.\\

Again the general statement follows easily from the codimension zero case (the special case in which $m=n$) by using normal disk bundles. We omit the argument. 

Let us show that Theorem \ref{symm} implies Theorem \ref{excision}. We may work in the codimension zero case. Let $N$, $M$, and $Q_i$ be as in Theorem \ref{excision} with $m=n$. For every $S$ we have a fibration
$$E(M-Q_S,N)\to E(M-Q_{\underline r},N).
$$
For any point $e\in E(M-Q_{\underline r},N)$ the fibers of these fibrations form a cube. The desired conclusion is equivalent to the assertion that for every choice of $e$ this cube of fibers is $(3{-}n{+}\Sigma_i(n{-}q_i{-}2))$-cartesian. But for every choice of $e$ we may write this cube of fibers as $E(Q_\bullet,N')$ where $N'$ is the closed complement of the image of $e$; thus the assertion follows from an instance of Theorem \ref{symm}. 

Conversely, Theorem \ref{symm} follows from Theorem \ref{excision}; in fact, any instance of Theorem \ref{symm} is related to an instance of Theorem \ref{excision} in the manner described above, for example by attaching an external collar $C$ to $N$ and letting $M$ be $C\cup Q_{\underline r}$. 

\subsection{A Formulation Using Moduli Spaces of Manifolds}

This version of the same result involves the following idea. For a smooth closed $(n{-}1)$-manifold $D$ we will define $\cal I(D)$, a moduli space for compact $n$-manifolds that have $D$ as boundary (\lq\lq interiors for $D$\rq\rq). The based loopspace of $\cal I(D)$ at $N$ is homotopy equivalent to the space of diffeomorphisms $N\to N$ fixed on the boundary $D$.
When $P$ has the same dimension as $N$ then the space $E(P,N)$ of codimension zero embeddings is equivalent to the homotopy fiber of a \lq\lq gluing-in-$P$\rq\rq\ map 
$$
\cal I\partial (N-P)\to \cal I(\partial N).
$$
This will be explained in detail in \S\ref{embeddings} and \S\ref{beginning}, including the relation with the results above, but here is the gist of it. Given such a $D$, and given $n$-dimensional triads $Q_1,\dots ,Q_r$ with the manifolds $\partial_0Q_i$ embedded disjointly in $D$, let $Q_S$ be the disjoint union $\cup_{i\in S}Q_i$ as before and let $D_S$ be the manifold that is obtained from $D$ by replacing $\partial_0Q_i$ with $\partial_1Q_i$ for each $i\in S$. Thus $D_S$ will be isomorphic to the boundary of the closed complement of any embedding of $Q_S$ in a manifold $N$ whose boundary is $D$. When $T\subset S$ there is a map $\cal I(D_S)\to\cal I(D_T)$ given by gluing in $Q_{S-T}$. This leads to an $r$-cube $\cal I(D_\bullet)$.

\begin{bigthm}\label{int}Let $D$ and $Q_i$ be as above with $r\ge 2$, and let $q_i$ be the handle dimension of $Q_i$ relative to $\partial_0Q_i$. Then the $r$-cube $\cal I(D_\bullet)$ is $(3{-}n{+}\Sigma_{i=1}^r(n{-}q_i{-}2))$-cartesian, except possibly in the case when $n=3$ and $q_1=\dots =q_r=1$.\end{bigthm}

This is equivalent to Theorem \ref{symm} because for any $N\in \cal I(D)=\cal I(D_\emptyset)$ the cube formed by the homotopy fibers of the maps $\cal I(D_S)\to \cal I(D)$ is homotopy equivalent to $E(Q_\bullet,N)$.

For the remainder of the paper we will sometimes use the abbreviation
$$
\Sigma=\Sigma_i(n{-}q_i{-}2).
$$

\subsection{The Second Main Result}

Where Theorem \ref{E} concerns spaces of embeddings alone, Theorem \ref{EF} compares spaces of embeddings with spaces of all (continuous or smooth) maps.

When $\partial_0P$ is embedded in $\partial N$, let $F(P,N)$ denote the space of continuous maps
from $P$ to $N$ that restrict to the given embedding on $\partial_0 P$. Thus there is a map $E(P,N)\to F(P,N)$. Given also submanifolds $Q_i$ as in Theorem \ref{E}, 
there is a map of $r$-cubes
$$
E(P,N - Q_\bullet) \to F(P,N - Q_\bullet).
$$
This map, regarded as an $(r{+}1)$-cube,
will be called $\EF(P,N - Q_\bullet)$.

According to Theorem \ref{E} the cube $E(P,N - Q_\bullet)$ is $(1{-}p{+}\Sigma)$-cartesian in all but the one exceptional case. Recall (Remark \ref{F}) that the cube $F(P,N - Q_\bullet)$ is $(1{-}p{+}\Sigma)$-cartesian in any case.

\begin{bigthm}\label{EF}Let $N$, $P$, and $Q_1,\dots ,Q_r$ be as Theorem \ref{E}. Then
the $(r{+}1)$-cube $\EF(P,N - Q_\bullet)$
is $(n{-}2p{-}1 {+}\Sigma_i(n{-}q_i{-}2))$-cartesian, except possibly in the case when $n{=}3$, $p{=}1$, and $q_i{=}1$ for all $i$.
\end{bigthm}

\begin{rem}
The $r {=} 0$ case of Theorem \ref{EF} is easy to prove by transversality. 
The $r{=}1$ case appears in a paper of
 Hatcher and Quinn \cite[th.~1.1]{Hatcher-Quinn} 
 (and \cite[th.\ 4.1]{Hatcher-Quinn} for
the families version; see also \cite[thm.~11.1]{Klein-Williams}). 
\end{rem}
\begin{rem}
A variant of Theorem \ref{EF} is obtained by replacing
each function space $F(P,N-Q_S)$ by the analogous space of smooth
immersions $I(P,N-Q_S)$. The conclusion of Theorem \ref{EF} is valid also for the $(r{+}1)$-cube $E(P,N-Q_\bullet)\to I(P,N-Q_\bullet)$, because by immersion theory the inclusion $I(P,N-Q_\bullet)\to F(P,N-Q_\bullet)$ is $\infty$-cartesian.
\end{rem}

It is clear that Theorem \ref{EF} implies Theorem  \ref{E} as long as $n{-}p\ge 2$ (since in that case $n{-}2p{-}1{+}\Sigma\ge1{-}p{+}\Sigma)$). We will also see that Theorem \ref{E} implies Theorem  \ref{EF}. It would be pleasant to simply prove one or the other of Theorems \ref{E} and \ref{EF} and then deduce the other from it. Instead, for various reasons, we will find ourselves needing to go back and forth between the two statements in the course of proving them both. Thus we will need to pay attention to which cases of Theorem \ref{E} imply which cases of Theorem \ref{EF} and vice versa.

\begin{rem}\label{leeway}If $n{-}p\ge 3$, so that $n{-}2p{-}1{+}\Sigma>1{-}p{+}\Sigma$, then we can say more. For one thing, the connectivity estimate for $E(P,N-Q_\bullet)$ must then be sharp if it is also sharp for $F(P,N - Q_\bullet)$ (as it often is). For another, we do not need the full strength of
Theorem \ref{EF} to deduce Theorem \ref{E} in that case. This will be useful in \S\ref{summing-up}.  \end{rem}

\begin{lem}\label{handlesplit}Theorem \ref{E} implies Theorem \ref{EF}. 
\end{lem}

\begin{proof}Without loss of generality $P$ is $n$-dimensional. Consider first the case when $P$ is a single $p$-handle, and argue by induction with respect to $p$. 

If $p=0$, so that $P$ may be taken to be a tubular neighborhood of a point, then $\EF(P,N-Q_\bullet)$ is $\infty$-cartesian. In fact, the fiber of $E(P,N-Q_S)\to F(P,N-Q_S)$ is equivalent to $O(n)$, independent of $S$. 

Now suppose that $p>0$, and assume the result for $p-1$. Decompose the handle $P=D^p\times D^{n-p}$ 
into three pieces $P_- \cup P_0 \cup P_+$ by cutting $D^p$ along two parallel $(p{-}1)$-planes. We obtain a square diagram of $(r{+}1)$-cubes
$$
\xymatrix{
\EF(P,N - Q_\bullet) \ar[r]\ar[d] & \EF(P_0 \cup P_+,N - Q_\bullet)\ar[d] \\
\EF(P_0 \cup P_-,N - Q_\bullet) \ar[r] & \EF(P_0,N - Q_\bullet)\, .
}
$$
The lower left
and upper right cubes consist of contractible spaces, because 
$P_0 \cup P_\pm$ is a collar on $\partial_0(P_0 \cup P_\pm)=(P_0\cup P_\pm)\cap\partial N$.
Both of these cubes are therefore $\infty$-cartesian. Since $P_0$ is a $(p{-}1)$-handle, the lower right hand cube is
$(n {-} 2(p{-}1) {-} 1 {+}\Sigma)$-cartesian, by
induction on $p$. It follows that the right hand arrow (an $(r{+}2)$-cube) is $(n {-} 2p {+}\Sigma)$-cartesian. 

The next claim is that the
$(r{+}3)$-cube given by the displayed square diagram
 is $(n{-}2p{-}1 {+}\Sigma)$-cartesian. This will imply that the left hand arrow is $(n{-}2p{-}1 {+}\Sigma)$-cartesian, and therefore that the same is true for the upper left cube.

To establish this claim, view the $(r{+}3)$-cube as a map from
$$
\xymatrix{
E(P,N - Q_\bullet) \ar[r] \ar[d] & E(P_0 \cup P_+,N - Q_\bullet)\ar[d]\\
E(P_0 \cup P_-,N - Q_\bullet) \ar[r] & E(P_0,N - Q_\bullet)\, .
}
$$
to 
$$
\xymatrix{
F(P,N - Q_\bullet) \ar[r] \ar[d] & F(P_0 \cup P_+,N - Q_\bullet)\ar[d] \\
F(P_0 \cup P_-,N - Q_\bullet) \ar[r]  & F(P_0,N - Q_\bullet)\, 
}.
$$
The second of these displayed squares (really $(r{+}2)$-cubes) is $\infty$-cartesian, because for each fixed index $S \subset
\underline r$ the corresponding square of spaces is $\infty$-cartesian. 
We use Theorem \ref{E} to show that the first square is $(n{-}2p{-}1 {+}\Sigma)$-cartesian. Fix a point in $E(P_0 \cup P_-,N - Q_{\underline r})$ and consider the fibers of the vertical maps in the square. This leads to an inclusion map of $r$-cubes
$$
E(P_+,N' - ( Q_\bullet \cup P_-)) \to 
E(P_+,N'- Q_\bullet) \, .
$$
Here $N'$ is the closed complement of the $(p-1)$-handle $P_0$ in $N$. Note that $P_-$ is a $p$-handle in $N'$ disjoint from each $Q_i$. Now apply Theorem \ref{E}, treating $P_-$ as one more submanifold $Q_{r+1}$. It follows that the $(r+1)$-cube above is $(1{-}p {+} \Sigma +(n{-}p{-}2))$-cartesian.

We complete the proof of the lemma by reducing to the case in which $P$ is a single handle, inducting on the number
of handles in a handle decomposition for $P$. Suppose that $P=A \cup H$ is the effect of attaching an index $p$ handle 
$H$ to $\partial_1 A$ (so that $H \cong D^{n-p} {\times} D^{p}$ meets $\partial_1 A$ transversely
at $\partial_0 H \cong D^{n{-}p} {\times} S^{p{-}1}$). We deduce the conclusion for $P$ from the conclusion for $A$. Consider the diagram of cubes
$$
\xymatrix{
E(A \cup H, N-Q_\bullet) \ar[r]\ar[d] & F(A\cup H, N- Q_\bullet) \ar[d] \\
E(A, N-Q_\bullet) \ar[r] &F(A, N- Q_\bullet)\, .
}
$$
Our assumption is that the $(r+1)$-cube defined
by the lower horizontal arrow is $(n{-}2p{-}1 {+} \Sigma)$-cartesian.
To prove the same connectivity statement for the upper horizontal
arrow, it is enough to consider the fibers of the vertical arrows and to
verify that the induced map of fibers is 
an $(n{-}2p{-}1 {+} \Sigma)$-cartesian  $(r{+}1)$-cube.
This has to be checked for every choice of basepoint in
$E(A,N - Q_{\underline r})$.

The fiber of the left-hand arrow is $E(H,N-(A \cup Q_\bullet)) $ and that of the right-hand arrow is $F(H,N-Q_\bullet)$. (Here the space $F(H,N-Q_S)$ may be a little unexpected; it consists of the maps $H\to N-Q_S$ with prescribed restriction to $\partial_0H$, but the prescribed map $\partial_0H\to N-Q_S$ does not go into the boundary.)
The map between them may be factored as
$$
E(H,N-(A\cup Q_\bullet)) \to 
F(H,N-(A\cup Q_\bullet))\to 
F(H,N-Q_\bullet)\, .
$$
The second of these maps is $(n{-}2p{-}1 {+} \Sigma)$-cartesian because
$$
N-(A\cup Q_\bullet)\to N-Q_\bullet
$$
is $(n{-}p{-}1 {+} \Sigma_i)$-cartesian
by the generalized Blakers-Massey theorem.
The first of them is $(n{-}2p{-}1 {+} \Sigma)$-cartesian by the $(n,p,p,q_1,\dots ,q_r)$ case of Theorem \ref{E}. 
\end{proof}

\begin{rem}\label{cases}The proof above shows that Theorem \ref{EF} holds for a given set of dimensions $(n,p,q_1,\dots ,q_r)$ provided that Theorem \ref{E} holds for 
$(n,p,p,q_1,\dots ,q_r)$ and more generally for $(n,p',p',q_1,\dots ,q_r)$ for all $p'\le p$. In contrast, the logical equivalence between Theorems \ref{E}, \ref{symm}, \ref{excision}, and \ref{int} was more straightforward, with never any change in the number of submanifolds or their handle dimensions.\end{rem}

\begin{rem}\label{calculus}The pattern of the inductive argument used in the first half of the proof above is one which lies at the heart of functor calculus. A very similar argument (again involving what might be called a downward induction on $r$) appears in the proof of the `First Derivative Criterion', Theorem 5.2 of \cite{Goodwillie_CALC2}. The first author long ago learned about the usefulness of splitting a $p$-handle into a $(p-1)$-handle and two $p$-handles from \cite{BLR}.\end{rem}

\subsection{More conventions}

The categories we consider are not small, so their nerves are not simplicial sets. For a brief discussion of some ways of working around this difficulty, see \cite{Goodwillie-Klein}.

We will not always 
distinguish between a category and its classifying space (= the realization of
its nerve). A functor $\cal A \to \cal B$
is said to be a weak equivalence if after taking 
realizations (of nerves) it becomes a homotopy equivalence. 
Similarly, it is $r$-connected if it becomes an $r$-connected map of 
spaces after realization.

\subsection{Outline}
 
Here is a schematic outline of the proof of the first main result in cases when all handle codimensions $n{-}p$ and $n{-}q_i$ are at least three. Consider the chain of 
forgetful maps
$$
\text{smooth} \quad \longrightarrow\quad
\text{block} \quad \longrightarrow \quad
\text{simple} \quad \longrightarrow \quad
\text{finite} \quad\longrightarrow \quad
\text{PD} 
$$
from smooth embeddings to block embeddings to simple Poincar\'e embeddings to finite Poincar\'e embeddings to Poincar\'e embeddings.  
The main result of \cite{Goodwillie-Klein} was an analogue of Theorem \ref{int} for Poincar\'e complexes. Here we easily deduce the corresponding statement for finite Poincar\'e complexes and then the corresponding statement for simple Poincar\'e complexes. We pass from the simple Poincar\'e statement
to the corresponding block statement using surgery theory
(compare \cite[th.\ 3.4.1]{GKW}).  Finally, to get to Theorem \ref{int} itself,
we use a multirelative connectivity statement from concordance theory, the main result of the first author's thesis \cite{thesis} (see \cite[\S 3.5]{GKW} for an outline
of two different proofs).  (Actually for this last step we switch from the point of view of  Theorem \ref{int} to that of Theorem \ref{symm}.)

In fact, the story
is a little more roundabout than this. One reason is that our Poincar\'e analogue of Theorem \ref{int} does not have the sharp estimate $3{-}n+\Sigma$; it is off by one dimension. We follow the scheme above to obtain
a weak (off by one) Theorem \ref{E}. From this we deduce 
a correspondingly weak Theorem \ref{EF}. From the weak Theorem \ref{EF} we get the sharp Theorem \ref{E}, and from that we get the sharp Theorem \ref{EF}. The other reason is that the surgery step requires manifolds to have dimension at least five, which makes for a bit of extra work in low dimensional cases.

Here is the organization of the paper.
\S\ref{embeddings} introduces the moduli spaces and spaces of embeddings that are the subject of the paper. In \S\ref{beginning} we lay out the five-step process outlined above. In \S\ref{PD-step} we pass from Poincar\'e complexes to finite Poincar\'e complexes. In \S\ref{simplestep} we pass from finite Poincar\'e complexes to simple Poincar\'e complexes. In \S\ref{surgery-step}
we pass from simple Poincar\'e complexes to the block world. In \S\ref{pseudoisotopy-step}
we go from block embeddings to smooth embeddings. In \S \ref{summing-up} we complete the proof of the main results in all cases where $n{-}p\ge 3$ and $n{-}q_i\ge 3$, making use of the result mentioned in Remark \ref{vw}. In \S\ref{cod2} we deal with the remaining cases. There are two appendices, one on Waldhausen's generalization
of Quillen's Theorems A and B, and the other on the obstruction
to making a finite Poincar\'e complex simple within its homotopy type.

\subsection{History of the results}For many years the first author believed (and stated) that he could prove most of this. But when he tried to write down a proof he found that some maps that should have been $k$-connected could only be shown to be (in the terminology of Lemma \ref{almost} below) \it almost\rm\ $k$-connected; a new idea was needed to obtain surjectivity on $\pi_0$. The second author came to the rescue with the homology truncation method, the main tool of \cite{Goodwillie-Klein}. 

\begin{acks} The second author is indebted to Bruce Williams for conversations 
about simple Poincare complexes and the block isotopy extension theorem. 
\end{acks}

\section{The Spaces}\label{embeddings}

Much of the time we will not be working directly with spaces of embeddings. Instead, we will translate statements about these into equivalent statements about certain moduli spaces. 

For a smooth closed manifold $D$ we will define the space $\cal I(D)$ of interiors.
We give an analogous definition of $\cal I^b(D)$ (the block space of interiors).
We also recall from \cite{Goodwillie-Klein} the definition of the space of interiors $\cal I^h(D)$ in the realm of Poincar\'e complexes, and we make analogous definitions for finite Poincar\'e complexes and for simple Poincar\'e complexes. 

\subsection{The manifold case}

Suppose that $D$ is a smooth closed manifold of dimension $n-1$.  We will be considering compact manifolds having boundary $D$, or more precisely compact manifolds equipped with a diffeomorphism between $D$ and the boundary. 
 
 \begin{defn} 
The simplicial groupoid $\cal I_\bullet(D)$ is defined as follows: 
in any simplicial degree $k$ the objects of $\cal I_k(D)$ are the compact
smooth $n$-manifolds $M$ having boundary $D$. A morphism $M \to M'$ is a diffeomorphism $M \times \Delta^k \to M' \times \Delta^k$ that commutes with projection to $\Delta^k$ and
restricts to the identity map on $D \times \Delta^k$.
\end{defn}

For an object $N$ the simplicial group consisting of the $\cal I_\bullet(D)$-morphisms from $N$ to $N$ will be denoted by $\Diff_\bullet(N)$.

\begin{rem} \label{moduli} (The nerve of) $\cal I_\bullet (D)$ is a classifying space for bundles of manifolds such that the fiberwise boundary is the trivial bundle with fiber $D$. The loopspace of $\cal I_\bullet(D)$ at an object $N$ may be identified
with the simplicial group $\Diff_\bullet(N)$. In fact, there is a contractible space $E_\bullet(N)$, the nerve of another simplicial groupoid, that fibers over $\cal I_\bullet (D)$ in such a way that the fiber over $N$ is $\Diff_\bullet(N)$.  This is an instance of a simple general fact about simplicial groupoids. The simplicial groupoid $E_\bullet(N)$ is defined as follows. An object in degree $k$ is an object of $\cal I_k(D)$ together with an $\cal I_k(D)$-isomorphism to $N$, and any two objects are uniquely isomorphic.
\end{rem}

Now suppose that $(N;\partial N)$ is a smooth compact $n$-manifold with boundary and $(P;\partial_0 P,\partial_1 P)$ is a smooth compact $n$-manifold triad. Fix a smooth embedding of $\partial_0 P$ in $ \partial N$, and identify $\partial_0P$ with its image.
We consider smooth embeddings $f\:P \to N$ that fix $\partial_0 P$ pointwise, and such that $f^{-1}(\partial N)=\partial_0P$.

\begin{defn}
The {\it  embedding space}  
$
E(P,N)
$
is the geometric realization of the simplicial set
whose $k$-simplices are families
of such embeddings parametrized smoothly by $\Delta^k$. Thus a $k$-simplex can be described as an embedding of $P {\times} \Delta^k$
in $N {\times} \Delta^k$ relative to $\partial_0 P \times \Delta^k$ that is compatible with projection
to $\Delta^k$. 
\end{defn}

The simplicial set
$E(P,N)$ is fibrant. \\

There is a variant of this definition in which a collar $\partial N\times I\subset N$ and a collar $\partial_0\times I\subset P$ are given and one allows only those embeddings  of $P$ in $N$ such that there exists a neighborhood of $\partial_0 P$ that is pointwise fixed. This gives a simplicial subset of $E(P,N)$ having the same
homotopy type. A reference for this material is \cite[app.~1]{BLR}.\\

We can describe the space $E(P,N)$ in terms of spaces of manifolds as follows: Given $N$ and $P$ as above, let $\partial_2 N$ be the closed complement of $\partial_0P$ in $\partial N$. Consider the closed $(n-1)$-manifold $\partial_1P\cup \partial_2N$. (It is not quite smooth in the usual sense: its two parts meet along a crease, as in the boundary of an $n$-manifold triad.) For any embedding of $P$ in $N$ relative to $\partial_0P$, the boundary of the closed complement of the image is $\partial_1P\cup \partial_2N$, so we will sometimes informally write the latter as $\partial(N-P)$ even if no embedding of $P$ in $N$ has been chosen. We claim that there is a fibration sequence up to homotopy 
$$
E(P,N)\to \cal I\partial(N-P)\to \cal I(\partial N).
$$
The map $\cal I\partial(N-P)\to \cal I(\partial N)$ is given by a functor that takes a manifold whose boundary is $\partial_1P\cup \partial_2N$ and glues it to $P$ along $\partial_1P$. The assertion is that the homotopy fiber over the vertex $N$ of $\cal I(\partial N)$ is weakly equivalent to $E(P,N)$. To prove it, we recall the contractible space $E(N)$ over $\cal I(\partial N)$ and argue that the fiber product 
$$
\cal I\partial(N-P)\times_{\cal I(\partial N)}E(N)
$$
is equivalent to $E(P,N)$. The fiber product is a simplicial groupoid (with different object sets in different dimensions) in which there is at most one map between any two objects. An object in degree $k$ consists of a manifold $X$ with boundary $\partial(N-P)$ together with a fiber-preserving diffeomorphism $(P\cup_{\partial_1P}X)\times \Delta^k\cong N\times \Delta^k$. The space $E(P,N)$ is isomorphic to a skeletal subcategory of this.  

\subsection{The block case}

Let $D$ be as above. We define the block analogue of $\cal I(D)$.

\begin{defn} 
The simplicial groupoid $\cal I^b_\bullet(D)$ is defined as follows: 
in any simplicial degree $k$ the objects of $\cal I_k(D)$ are the compact
smooth $n$-manifolds $M$ having boundary $D$. A morphism $M \to M'$ is a diffeomorphism $M \times \Delta^k \to M' \times \Delta^k$ that preserves each face $M\times \partial_j\Delta^k$ and restricts to the identity map on $D \times \Delta^k$.
\end{defn}

\begin{rem}
It is clear how to define face operators. Degeneracy operators may appear problematic, but in fact there is no difficulty. See pages 120-121 of \cite{BLR}.
\end{rem}

For an object $N$ we have its simplicial group of automorphisms, the block diffeomorphism group $\Diff^b_\bullet(N)$. By the same reasoning as in the ordinary manifold case, using a contractible space $E^b(N)$, this is equivalent to the loop space of $\cal I^b_\bullet(D)$ at $N$.\\

For $N$ and $P$ as above, we also have the block analogue of $E(P,N)$:

\begin{defn}
The {\it  block embedding space}  
$
E^b(P,N)
$
is the geometric realization of the simplicial set
whose $k$-simplices are families
of face-preserving embeddings $P\times \Delta^k\to N\times \Delta^k$ fixed on $\partial_0P$. 
\end{defn}

Like
$E(P,N)$, this is fibrant.

\begin{rem}The inclusion 
 $E(P,N) \to E^b(P,N)$ induces a surjection on $\pi_0$ but it is rarely
a homotopy equivalence. In fact, in general it does not give a bijection on $\pi_0$; two embeddings can be concordant without being isotopic. In the codimension three case that we will mostly be concerned with here (when the handle dimension of $P$ relative to $\partial_0P$ is at most $n-3$), it does in fact give a bijection of $\pi_0$, by a theorem of Hudson \cite{Hudson}. (See also \S\ref{pseudoisotopy-step} below.)
\end{rem}

We cannot assert a fibration sequence
$$
E^b(P,N)\to \cal I^b\partial(N-P)\to \cal I^b(\partial N)
$$
in general, but we can do so when the handle codimension of $P$ is at least three. The issue is that the complement of a face-preserving embedding $P\times \Delta^k\to N\times \Delta^k$ might not be diffeomorphic to anything of the form $X\times \Delta^k$. But in the codimension three case it must be, because Hudson's theorem implies that every face-preserving embedding $P\times \Delta^k\to N\times \Delta^k$ is isotopic through face-preserving embeddings $P\times \Delta^k\to N\times \Delta^k$ to the product of an embedding $P\to N$ and the identity.

\subsection{The Poincar\'e duality case}

In \cite{Goodwillie-Klein} we defined $\cal I^h(D)$, the analogue of $\cal I(D)$ for Poincar\'e  complexes, and showed that its loopspace at an object $N$ is equivalent to the space of homotopy equivalences $N\to N$ fixed on the boundary. We used a fibration sequence
$$
E^h(P,N)\to \cal I^h\partial(N-P)\to \cal I^h(\partial N)
$$
as definition of the space of Poincar\'e embeddings. Let us recall the construction.\\

If $j\: A\to B$ is a morphism in the category $\cal T$ of topological spaces, let $\cal T(j)$
be the category of factorizations of $j$.  An object of $\cal T(j)$
consists of a space $X$ together with maps $i\:A\to X$ and $p\: X \to B$ such
that $p\circ i = j$. A morphism $(X,i,p) \to (X',i',p')$ is a map $f\: X\to X'$
such that $f\circ i = i'$ and $p'\circ f = p$. We sometimes write $X$ for $(X,i,p)$,
and we write $\cal T(A\to B)$ when $j$ is understood. A morphism is called a weak equivalence
(cofibration, fibration) if it is a weak equivalence (cofibration, fibration) of spaces.
This is a model structure. Let $w\cal T(j) \subset \cal T(j)$ be the subcategory having all of the objects but having only the weak equivalences as morphisms.

\begin{defn} If $D$ is a Poincar\'e space of formal dimension $n-1$, then 
$\cal I^h(D)$ is the full subcategory of  $w\cal T(D \to \ast)$ consisting of those
objects $X$ such that $D \to X$ satisfies relative $n$-dimensional duality.
(This means that $(\bar X,D)$ is a Poincar\'e pair, where $\bar X$ is the mapping cylinder
of $D\to X$. There is no finiteness or simple-homotopy requirement in the definition of Poincar\'e pair.) 
\end{defn}

According to Lemma 2.7 of \cite{Goodwillie-Klein}, the loopspace $\Omega_X\cal I^h(D)$ at the vertex $X$ is homotopy equivalent to the space $\Aut^h(X)$ of all homotopy equivalences $X\to X$ that are fixed on the boundary $D$, as long as $D\to X$ is a cofibration.

If $(P;\partial_0P,\partial_1P)$ and $(N;\partial_0P,\partial_2N)$ are CW Poincar\'e triads of the same formal dimension, then there is a gluing-in-$P$ functor 
$$
\mathcal I^h\partial(N-P):=\mathcal I^h(\partial_1P\cup\partial_2N)\to \mathcal I^h(\partial N).
$$

Although we do not really need it here, we also recall:

\begin{defn} If $(P;\partial_0P,\partial_1P)$ and $(N;\partial_0P,\partial_2N)$ are CW Poincar\'e triads of the same formal dimension, then the Poincar\'e embedding space $E^h(P,N)$ is defined as the homotopy fiber of $\mathcal I^h\partial(N-P)\to \mathcal I^h(\partial N)$ with respect to the vertex $N$. 
\end{defn}

\subsubsection{The finite variant}\label{finite}

If the Poincar\'e space $D$ is a finite CW complex, then we define $\cal I^f(D) \subset
\cal I^h(D)$ to be the full subcategory whose objects are finite Poincar\'e pairs $(X,D)$. By a finite Poincar\'e pair we just mean a finite CW pair that satisfies relative Poincar\'e duality.

Clearly the nerve of $\cal I^f(D)$ is equivalent to the nerve of the full subcategory of $\cal I^h(D)$ whose objects are finite up to homotopy (weakly equivalent in $\cal T(D\to *)$ to finite Poincar\'e pairs). Thus the inclusion map $\cal I^f(D) \to
\cal I^h(D)$ is essentially the inclusion of an open and closed subset, and its homotopy fiber with respect to an object $X\in \cal I^h(D)$ is either contractible (if $X$ is finite up to homotopy) or empty (if it is not). 

If the triads $P$ and $N$ are finite then the gluing functor \newline
$\mathcal I^h(\partial(N-P))\to \mathcal I^h(\partial N)$ takes $\mathcal I^f(\partial(N-P))$ into $ \mathcal I^f(\partial N)$. We might again refer to its homotopy fiber as an embedding space and denote it by $E^f(P,N)$. It is equivalent to a union of components of $E^h(P,N)$, and it will follow from \S\ref{PD-step} that when the handle codimension of $P$ is at least three then it is equivalent to all of $E^h(P,N)$.

\subsubsection{The simple homotopy variant}

Now suppose that the finite Poincar\'e complex $D$ is simple. Let $\cal I^s(D)\subset \cal I^f(D)$ have for objects the simple Poincar\'e pairs $(X,D)$ and for morphisms the simple homotopy equivalences. (For details concerning simple Poincar\'e pairs, see the Appendix \S\ref{simple}.)

Just as $\Omega_X\cal I^h(D)$ is equivalent to the space $\Aut^h(X)$ of homotopy equivalences, $\Omega_X\cal I^s(D)$ is equivalent to the open and closed subspace $\Aut^s(X)$ of simple homotopy equivalences. Thus $\cal I^s(D)$ is essentially a covering space of $\cal I^f(D)$; the forgetful map $i:I^s(D)\to I^f(D)$ gives an isomorphism $\pi_j(\cal I^s(D),X)\to \pi_j(\cal I^s(D),X)$ for $j\ge 2$ and an injection for $j{=}1$.

In other words, for any $X\in \cal I^f(D)$ the homotopy fiber of $i$ with respect to $X$ is homotopically discrete. To describe the components of the fiber we use Quillen's Theorem B.  The homotopy fiber is equivalent to the left fiber $i/ X$, since a weak equivalence $X\to X'$ fixed on $D$ induces a weak equivalence of left fibers (see \cite{Goodwillie-Klein}). The components of the left fiber correspond canonically with the equivalence classes (if any) of simple finiteness structures on $(X,D)$ in the sense of the Appendix \S\ref{simple}.

The lemma below is proved in the Appendix. Let $\Wh(X) = \Wh(\pi_1(X))$ be the Whitehead group. Define the
norm map
$$\frak N\: \Wh(\pi_1(X))\to \Wh(\pi_1(X))
$$
by $\frak N(x)= x + (-1)^n x^*$, where $x\mapsto x^*$ is the 
canonical involution twisted by the orientation bundle. Consider the Whitehead torsion $\tau(X,D)\in \Wh(X)$ of the duality map.
\begin{lem} \label{I^s}  
Let $(X,D)$ be a finite Poincar\'e pair and assume that the boundary $D$ is simple. Then $(X,D)$ admits a simple structure if and only if $\tau(X,D)$ belongs to the image of $\frak N$. When a simple structure exists, the set of equivalence classes of such structures has a natural free transitive action of the kernel of $\frak N$. \end{lem}

If the finite triads $P$ and $N$ are simple then restriction of the gluing functor $\mathcal I^f(\partial(N-P))\to \mathcal I^f(\partial N)$ gives a map $\mathcal I^s(\partial(N-P))\to  \mathcal I^s(\partial N)$. We might refer to its homotopy fiber as an embedding space and denote it by $E^s(P,N)$. The map $E^s(P,N)\to E^f(P,N)$ is essentially a covering space, and we will see in \S\ref{simplestep} that in the codimension three case it is an equivalence.

\subsection{Comparison maps}

We need to use maps
$$
\cal I_\bullet(D) \to \cal I_\bullet^b(D)\to\cal I^s(D) \to \cal I^f(D) \to \cal I^h(D).
$$
The first arrow and the last two are inclusion maps that have already been mentioned. There are two little difficulties in defining an \lq\lq inclusion\rq\rq\ $\cal I_\bullet^b(D)\to\cal I^s(D)$.

The first issue is that a smooth manifold does not have a preferred cell structure. But it does have a preferred equivalence class of finiteness structures (in the sense of Appendix \S\ref{simple}), since every smooth triangulation gives a cell structure and since the identity map provides a simple homotopy equivalence between any two of these. Introduce an equivalent but larger version of $\cal I^s(D)$: An object is now a Poincar\'e pair $(N,D)$ together with an equivalence class of finiteness structure such that the resulting finite Poincar\'e pair is simple. A morphism is a homotopy equivalence of pairs (restricting to the identity on $D$) such that the resulting equivalence between finite complexes is simple. \\

The second issue is that $\cal I^s(D)$ is a category (simplicial set) rather than a simplicial category (bisimplicial set). This is easily remedied by introducing another simplicial direction:

\begin{defn} $\cal I^{s}_\bullet(D)$ is the simplicial category
which in every simplicial degree $k$ has the same objects as $\cal I^s(D)$, and in which
a morphism $X \to X'$ is a simple homotopy equivalence $X \times \Delta^k \to X'\times \Delta^k$ that preserves the projection to $\Delta^k$ and restricts to the identity on $D \times \Delta^k$.  The simplicial category $\cal I^{bs}_\bullet(D)$ is the block version of this, in which morphisms are required to preserve faces but not the projection.
\end{defn}

\begin{lem} The inclusion functors $\cal I^s(D) \to \cal I^s_\bullet (D)\to\cal I^{bs}_\bullet(D)$
are weak equivalences.
\end{lem} 

\begin{proof}
For the first inclusion see for example \cite{Goodwillie-Klein}. For the second, note that for any objects $X$ and $X'$ the inclusion of simplicial sets $\hom_{\cal I^s_\bullet(D)}(X,X')\to \hom_{\cal I^{bs}_\bullet}(D)(X,X')$ is a weak equivalence. (The latter is isomorphic to the product of the former with another simplicial set, and this other factor is contractible because the space of face-preserving continuous maps $\Delta^k\to \Delta^k$ is convex.)
\end{proof}

Now there is an inclusion map $\cal I^b_\bullet(D)\to \cal I^{bs}_\bullet(D)\sim \cal I^s(D)$. Simplifying the notation, we write
$$
\cal I(D) \to \cal I^b(D)\to\cal I^s(D) \to \cal I^f(D) \to \cal I^h(D).
$$

\section{Beginning of the Main Proof\label{beginning}}

We now begin proving Theorem \ref{int} in the case when all handle codimensions are at least three.

Thus suppose we have a closed $(n-1)$-manifold $D$, a finite collection of $n$-dimensional compact manifold triads
$(Q_1,\dots,Q_r)$, and disjoint embeddings of the $\partial_0 Q_i$ in $D$.

For each $S\subset \underline r$ define the triad $Q_S$ to be the disjoint union $\cup_{i\in S} Q_i$.
Let $D_S$ be the closed $(n-1)$-manifold obtained from $D$ by replacing $\partial_0 Q_i$ with $\partial_1 Q_i$ for each $i \in S$.  
If $T\subset S$ then we
have a map $\cal I^h(D_S) \to\cal I^h(D_T)$ given by gluing in $Q_{S-T}$. 
The resulting $r$-cubical diagram $S \mapsto \cal I^h(D_S)$ will be denoted
by $\cal I^h(D_\bullet)$. (The diagram does not strictly commute, but we gave a preferred method in \cite{Goodwillie-Klein} for rectifying
the cube to a strictly commutative one by replacing each $\cal I^h(D_S)$ by something weakly equivalent.
In the interest of clarity of exposition we will ignore that detail here.)\\

The same can be done with the spaces $\cal I(D_S)$, and with everything in between; we have maps of cubes
$$
\cal I(D_\bullet) \to \cal I^b(D_\bullet) \to \cal I^s(D_\bullet) \to \cal I^f(D_\bullet) \to 
\cal I^h(D_\bullet) \, .
$$
The next few sections of this paper are devoted to proving that, as long as $n{-}q_i\le 3$ for all $i$, the $r$-cube $\cal I(D_\bullet)$ is $(2{-}n{+}\Sigma)$-cartesian. Later this number will be improved to $(3{-}n{+}\Sigma)$. \\

The starting point is the main result of \cite{Goodwillie-Klein}:

\begin{thm}[cf.\ {\cite[th.\ 7.1]{Goodwillie-Klein}}]If $D$ and $Q_1,\dots ,Q_r$ are as above, with $r\ge 2$ and $n{-}q_i\ge 3$ for all $i$, then the $r$-cube $\cal I^h(D_\bullet)$ is $(2{-}n{+}\Sigma)$-cartesian.
\end{thm}

To conclude that $\cal I(D_\bullet)$ is $(2{-}n{+}\Sigma)$-cartesian under the same hypotheses, it is then enough to show that each of the four arrows above, regarded as an $(r{+}1)$-cube, is $(2{-}n{+}\Sigma)$-cartesian. This will be done in the next four sections. In fact, we will find that $\cal I(D_\bullet) \to \cal I^b(D_\bullet)$ is $\Sigma$-cartesian and that the other three are $\infty$-cartesian. \\

\section{The Finiteness Step \label{PD-step}}

\begin{lem}
If $D$ and $Q_1,\dots ,Q_r$ are as in \S\ref{beginning}, with $r\ge 1$ and with $n{-}q_1\ge 3$, then the map $\cal I^f(D_\bullet)\to \cal I^h(D_\bullet)$, regarded as an $(r{+}1)$-cube, is $\infty$-cartesian. 
\end{lem}

\begin{proof}
By repeatedly using the fact that a cube is $\infty$-cartesian if it can be viewed as a map between two $\infty$-cartesian cubes, it is enough to show that for each $S\subset\lbrace 2,\dots,r\rbrace$ the square 
$$
\xymatrix{
\cal I^f(D_{S\cup\lbrace 1\rbrace}) \ar[r] \ar[d] & \cal I^h(D_{S\cup\lbrace 1\rbrace})\ar[d] \\
\cal I^f(D_S) \ar[r] & \cal I^h(D_S) 
}
$$
is $\infty$-cartesian. In other words, it suffices to consider the case when $r=1$.\\

Thus suppose we have a single triad $Q$ with handle dimension at most $n{-}3$ and an embedding $\partial_0Q\subset D$, and write $D'$ for the result of substituting $\partial_1Q$ for $\partial_0Q$ in $D$. We have to show that for every Poincar\'e space $X$ with boundary $D$ the canonical map of homotopy fibers
$$ 
\fib_X(\cal I^f(D')\to \cal I^h(D')\to \fib_{X\cup Q}(\cal I^f(D)\to \cal I^h(D)
$$
is a weak equivalence.

As noted in \ref{finite}, these homotopy fibers are either empty or weakly contractible. We have to rule out the possibility that the first is empty and the second is not. That is, we must show that if $X\cup Q$ is homotopically finite then $X$ is homotopically finite. 

This uses the handle dimension assumption, which insures that the pair $(Q,X\cap Q)=(Q,\partial_1 Q)$ is $2$-connected. The map $X\to X\cup Q$ is then also $2$-connected. In particular, it induces an isomorphism of fundamental groups. The Wall finiteness obstruction for $X$ vanishes because it can be identified with the corresponding obstruction for $X\cup Q$.
\end{proof}

\section{The Simple Homotopy Step}\label{simplestep}

\begin{lem}
If $D$ and $Q_1,\dots ,Q_r$ are as in \S\ref{beginning}, with $r\ge 1$ and with $n{-}q_1\ge 3$, then the map $\cal I^s(D_\bullet)\to \cal I^f(D_\bullet)$, regarded as an $(r{+}1)$-cube, is $\infty$-cartesian. 
\end{lem}

\begin{proof}Again it suffices to consider the $r{=}1$ case. In the notation of the previous section, we must show that for every vertex $X$ of $\cal I^f(D)$ the canonical map of homotopy fibers
\begin{equation}\label{Q} 
\fib_X(\cal I^s(D')\to \cal I^f(D'))\to \fib_{X\cup Q}(\cal I^s(D)\to \cal I^f(D))
\end{equation}
is a weak equivalence, as long as $Q$ has handle dimension at most $n{-}3$. \\

Recall that these homotopy fibers are homotopically discrete spaces, so that we have only to show that the gluing-in-$Q$ map \eqref{Q}
induces a bijection of component sets.
The $2$-connected map $X\to X\cup Q$ induces an isomorphism of Whitehead groups, and it is clear by naturality that the torsion of the duality map of $X\cup Q$ corresponds to that of $X$. In particular the latter is in the image of $\frak N$ if and only if the former is. Thus by Lemma \ref{I^s} the one fiber is nonempty if and only if the other is nonempty. Furthermore, when this holds then the canonical map from the component set of the one to the component set of the other commutes with the (free, transitive) action of the kernel of $\frak N$ and so must be a bijection.
\end{proof}

\section{The Surgery Step \label{surgery-step}}

\begin{lem}\label{surgerylemma}
If $D$ and $Q_1,\dots ,Q_r$ are as in \S\ref{beginning}, with $r\ge 2$, $n{-}q_1\ge 3$, and $n\ge 5$, then the map $\cal I^b(D_\bullet)\to \cal I^s(D_\bullet)$, regarded as an $(r{+}1)$-cube, is $\infty$-cartesian. 
\end{lem}

The proof will be a standard application of unobstructed surgery theory. 

\subsection{Sketch of the argument\label{surgery-sketch}}
The plan is to show that for every choice of basepoints in $\cal I^s(D_\bullet)$ the $r$-cube made up of the homotopy fibers of the maps $\cal I^b(D_S)\to \cal I^s(D_S)$ is $\infty$-cartesian. \\

For an object $X$ of $\cal I^s(D)$, {i.e.} a simple Poincar\'e complex whose boundary is the manifold $D$, we will denote the homotopy fiber of the map $\cal I^b(D) \to \cal I^s(D)$ by $\cal S(X)$ and interpret it as a space of manifold structures (relative to $D$) on the pair $(X,D)$. Thus in our situation if $X_\bullet$ is a family of compatible basepoints for $\cal I^s(D_\bullet)$ then the cube of homotopy fibers mentioned above can be written $\cal S(X_\bullet)$. \\

The space $\cal S(X)$ of manifold structures (or solved surgery problems) on $(X,D)$ has a canonical map, the normal invariant, to the space $\cal N(X)$ of normal structures (or surgery problems). 
In our situation the normal invariant gives a map of $r$-cubes $\cal S(X_\bullet)\to \cal N(X_\bullet)$. \\

Wall's $\pi$-$\pi$ Theorem basically says that the obstruction to solving a surgery problem is unchanged when the fundamental group of $X$ is unchanged. Using it we will see that when $X$ and $X\cup Q$ differ by handles of co-index at least three then the square
$$
\xymatrix{
\cal S(X) \ar[r] \ar[d] & \cal S(X\cup Q)\ar[d] \\
\cal N(X) \ar[r] & \cal N(X \cup Q)
}
$$
is homotopy cartesian.
It follows that in our situation the $(r{+}1)$-cube $\cal S(X_\bullet)\to \cal N(X_\bullet)$ is $\infty$-cartesian if $r$ is at least one. Thus in order for the cube $\cal S(X_\bullet)$ to be $\infty$-cartesian it suffices if the cube $\cal N(X_\bullet)$ is $\infty$-cartesian, which it rather obviously is as long as $r$ is at least two.
Here are some details.

\subsection{The space of manifold structures} 
Suppose that $D$ is a closed smooth $(n{-}1)$-manifold.
For an object  $X\in \cal I^s(D)$, let $\cal S(X)$ be the homotopy
fiber at $X$ of $\cal I^b(D) \to \cal I^s(D)$.
Using Waldhausen's Theorem B' (recalled in \S\ref{WaldAB} below) one may describe $\cal S(X)$ as the nerve of the following simplicial category: In simplicial degree $k$ its 
objects are given by pairs $(M,f)$ in which $M$ is a compact smooth manifold whose boundary is 
$D$ and $f\: M \times \Delta^k \to X \times \Delta^k$ is a block simple homotopy equivalence restricting to the identity on $D\times \Delta^k$. A morphism
$(M,f) \to (M',f')$ is a block diffeomorphism $h\: M \times \Delta^k \to M'\times \Delta^k$
such that $f'\circ h = f$.

\begin{rem} Weiss and Williams \cite[p.\ 168]{Weiss-Williams_auto} define the structure space of $X$
as the realization of the simplicial category which in simplicial degree $k$ has objects
$(M,f)$ where $M$ is a compact smooth manifold and $f\: M \times \Delta^k \to X \times \Delta^k$
is a block simple homotopy equivalence that restricts to a block homeomorphism $\partial M \times \Delta^k \to D \times \Delta^k$. A morphism $(M,f) \to (M',f')$ is a
block diffeomorphism $h\: M \times \Delta^k \to M' \times \Delta^k$ which 
commutes with the reference maps to $X \times \Delta^k$. Our structure space is therefore
a subspace of the Weiss-Williams one. However, a straightforward application of 
Waldhausen's Theorem A' shows that the inclusion is a weak equivalence. 
\end{rem}

\subsection{The space of normal structures} Let $X$ be a Poincar\'e space whose boundary $D$
is a compact smooth $(n{-}1)$-manifold. Let $\eta$ be the Spivak normal spherical fibration of $X$. Its restriction to $D$ is
canonically equivalent to the underlying stable spherical fibration of the stable normal bundle of $D$.  

A {\it normal structure} on $X$ is
a stable vector bundle $\xi$ on $X$ restricting to the stable normal bundle of $D$, together with an equivalence of stable spherical fibrations between $\eta$ and the underlying spherical fibration of $\xi$. The equivalence is required to restrict to the given equivalence on $D$. 

The space $\cal N(X)$ of normal structures is defined by letting a $k$-simplex be a stable vector bundle $\xi$ on $X\times \Delta^k$ restricting to the stable normal bundle of $D\times \Delta^k$, together with an equivalence of stable spherical fibrations over $X\times \Delta^k$ restricting to the given equivalence on $D\times \Delta^k$. This is equivalent to the singular complex of a homotopy fiber of
$$
\text{\rm map}(X,BO) \to \text{\rm map}(X,BG) \times_{ \text{\rm map}(D,BG)} \text{\rm map}(D,BO),
$$
where $BO$ and $BG$ classify stable vector bundles and stable spherical fibrations respectively.  

\begin{rem}A normal structure determines a surgery problem, well-defined up to cobordism. That is, there is always a smooth compact manifold $M$ with boundary $D$ and a map $f:M\to X$ covered by a stable vector bundle isomorphism between $f^*\eta$ and the normal bundle, with both $f$ and the bundle isomorphism being the identity on $D$ and with $f$ carrying the fundamental class of $(M,D)$ to that of $(X,D)$. In fact, one could define a simplicial set that is equivalent to $\cal N(X)$ by taking surgery problems as the $0$-simplices, cobordisms between surgery problems as the $1$-simplices, and so on.\end{rem}

\subsection{The normal invariant} 
The {\it normal invariant} is  a (weak) map $\cal S(X) \to \cal N(X)$. To define it, we introduce a simplicial category $\cal S'(X)$ equivalent to $\cal S(X)$ and give a map $\cal S'(X)\to \cal N(X)$. 
An object of $\cal S'(X)$ consists of an object $(M,f)$ of $\cal S(X)$, a stable vector bundle $\phi$ on $X$, and an identification between $f^*(\phi) \in \cal N(M)$ and the stable normal bundle of $M$.
A morphism $(M,f,\phi) \to (M',f',\phi')$ consists of a $\cal S(X)$-morphism $h\:(M,f) \to (M',f')$
together with a compatible bundle isomorphism $\phi' \to \phi$.  The forgetful map $\cal S'(X) \to \cal S(X)$ is a weak equivalence by Waldhausen's Theorem A' (cf.\ \ref{A'}). 
The forgetful map $(M,f,\phi) \mapsto \phi$ gives a functor 
$$
\cal S'(X)\to \cal N(X)\, .
$$
The weak map $\cal S(X) \overset{\sim}\leftarrow \cal S'(X) \to \cal N(X)$ is the normal invariant.

\begin{rem} Let $F(X,G/O)$ be the function space of maps $X\to G/O$ which take $D$ to the basepoint. It is an $H$-space, and it acts on $\cal N(X)$ making it into a $F(X,G/O)$-torsor.
In particular, a choice of basepoint in $\cal N(X)$
determines a homotopy equivalence $\cal N(X) \simeq
F(X,G/O)$.  
\end{rem}

\begin{proof}[Proof of Lemma \ref{surgerylemma}]
Fix any vertex $X$ of $\cal I^s(D_{\underline r})$ and thus a collection $X_\bullet$ of compatible base points for $\cal I^s(D_\bullet)$. The homotopy fibers of $\cal I^b(D_S)\to \cal I^s(D_S)$ form an $r$-cube $\cal S(X_\bullet)$. The problem is to show that it is $\infty$-cartesian.\\

We first prove that the $(r{+}1)$-cube $\cal S(X_\bullet) \to \cal N(X_\bullet)$ is
$\infty$-cartesian. Using the same principle as in \S\ref{PD-step}, we can reduce to the $r{=}1$ case. Let $D$, $Q$, and $D'$ be as in \S\ref{PD-step}.
For a simple Poincar\'e pair $(X,D')$ we have the square
$$
\xymatrix{
\cal S(C) \ar[r] \ar[d] & \cal S(C\cup Q)\ar[d] \\
\cal N(C) \ar[r] & \cal N(C \cup Q)
}.
$$
\begin{prop}\label{pi-pi}
If $n\ge 5$ and the handle dimension of $Q$ is at most $n{-}3$, then the square above is $\infty$-cartesian. 
\end{prop}
\begin{proof}
We first show that it
is $0$-cartesian, that is, that any point in the homotopy limit of 
$$\cal S(C\cup Q) \to \cal N(C\cup Q) 
\leftarrow \cal N(C) $$
can be deformed into $\cal S(C)$.  We can assume that the point
is described by a manifold structure on $C\cup Q$, 
a normal structure $\phi$ on $C$, and a $1$-simplex in $\cal N(C \cup Q)$ which connects them.   
The $1$-simplex gives rise to a surgery problem over $(C\cup Q) \times I$ that is already solved on all of the boundary except $C\times 1$, i.e., on
$(C\cup Q) \times 0 \cup \partial (C \cup Q) \times I \cup Q \times 1$. What we need is to extend the solution to all of $(C\cup Q)\times I$.  The inclusion $C \times 1 \subset (C\cup Q) \times I$ is $2$-connected because of the condition $n{-}q\le 3$, and the dimension $n{+}1$ is at least six, so such an extension exists by the $\pi$-$\pi$ theorem (\cite[th.\ 3.3]{Wall}).
Consequently, the square is
$0$-cartesian.\\

Now let $F$ be any homotopy
fiber of the map 
$$
\cal S(C) \to \holim (\cal S(C\cup Q) \to \cal N(C\cup Q) 
\leftarrow \cal N©).
$$
We have just shown that $F$ is nonempty. To see that it is weakly contractible, we must show that for $m\ge 1$ any map $S^{m-1} \to F$ can be extended to $D^m$.
This means having to solve a surgery problem over $D^m \times (C\cup Q) \times I$ relative to 
$(S^{m-1} \times (C\cup Q) \times I) \cup (D^m \times C \times 1)$. Again by the $\pi$-$\pi$ Theorem this can be done.
\end{proof}

It remains to see that the $r$-cube $\cal N(X_\bullet)$ is $\infty$-cartesian. In the $r$-cube $X_\bullet$ 
every two-dimensional face is a homotopy pushout. It follows (interpreting $\cal N(-)$ as a space of lifts from $BG$ to $BO$) that each two-dimensional face of  $\cal N(X_\bullet)$ is a homotopy pullback. In particular $\cal N(X_\bullet)$ is $\infty$-cartesian.  

\end{proof}

\subsection{The $4$-dimensional case}

The hypothesis $n\ge 5$ was needed in the first application of the $\pi$-$\pi$ theorem in the proof of \ref{pi-pi}. Thus when $n=4$ we can no longer say that the map
$$
\cal S(C) \to \holim (\cal S(C\cup Q) \to \cal N(C\cup Q) 
\leftarrow \cal N(C))
$$
is surjective on $\pi_0$. We can still say that it is injective on $\pi_0$ and bijective on homotopy groups, because the second application of the $\pi$-$\pi$ theorem required only $m{+}n\ge 5$ for $m\ge 1$. To record and exploit this information, we introduce some language for discussing connectivity of cubes when $\pi_0$-surjectivity may be lacking.

Recall that a space $X$ is called $k$-connected if for every $m$ with $-1\le m\le k$ every (continuous) map $S^m\to X$ can be extended to $D^{m+1}$, and that a map $X\to Y$ is $k$-connected if for every point in $Y$ the homotopy fiber of the map is a $(k-1)$-connected space, and that a cubical diagram $X_\bullet$ is $k$-cartesian if the associated map $X_\emptyset\to \holim_{S\neq \emptyset}X_S$ is $k$-connected.

\begin{defn}\label{almost}A space is {\it almost $k$-connected} if it is either empty or $k$-connected. A map of spaces is {\it almost $k$-connected} if each of its homotopy fibers is almost $(k-1)$-connected. A cube $X_\bullet$ of spaces is {\it almost $k$-cartesian} if the associated map $X_\emptyset\to \holim_{S\neq\emptyset}X_S$ is almost $k$-connected.
\end{defn}

Thus a map is $k$-connected if it is both $0$-connected and almost $k$-connected, and a cube is $k$-cartesian if it is both $0$-cartesian and almost $k$-cartesian.

Note that if $k\ge 1$ then an almost $k$-connected map can also be described as a map that induces a surjection of $\pi_k$ and an isomorphism of $\pi_m$ for $0<m<k$, for all basepoints in the domain, and an injection (but not necessarily a surjection) of $\pi_0$.

\begin{add}\label{4}
With the same hypotheses as Lemma \ref{surgerylemma} except that $n=4$, the cube $\cal I^b(D_\bullet)\to \cal I^s(D_\bullet)$ is almost $\infty$-cartesian. \\
\end{add}
 
\begin{proof}The steps are just as in the proof of Lemma \ref{surgerylemma}. We know that Proposition \ref{pi-pi} is valid in the $n{=}4$ case with a weakened conclusion: the square is almost $\infty$-cartesian. To deduce the correspondingly weakened version of \ref{surgerylemma}, we need the general statements appearing below in Lemma \ref{lem:helper}. The proofs, which are straightforward modifications of proofs of the corresponding statements without the \lq almost\rq , are left to the reader.
\end{proof}

\begin{lem} \label{lem:helper}
Let $X_\bullet\to Y_\bullet$ be a map of $r$-cubes, viewed as an $(r+1)$-cube.
Then
\begin{enumerate}
\item $X_\bullet\to Y_\bullet$ is almost $k$-cartesian if for every choice of compatible basepoints in $Y_\bullet$ the $r$-cube $\fib(X_\bullet\to Y_\bullet)$ is almost $k$-cartesian.

\item 
$X_\bullet$ is almost $k$-cartesian if $X_\bullet\to Y_\bullet$ and $Y_\bullet$ are almost $k$-cartesian.

\item 
For composable maps of spaces, if $f\circ g$ is almost $k$-connected and $g$ is almost $(k+1)$-connected then $f$ is almost $k$-connected.

\item
$X_\bullet\to Y_\bullet$ is almost $k$-cartesian if $X_\bullet$ is almost $k$-cartesian and $Y_\bullet$ is almost $(k+1)$-cartesian.
\end{enumerate}
\end{lem}

\section{The Concordance Step \label{pseudoisotopy-step}}

\begin{lem}\label{conclemma}
If $D$ and $Q_1,\dots ,Q_r$ are as in \S\ref{beginning}, with $r\ge 1$ and $n-q_i\ge 3$ for all $i$, then the map $\cal I(D_\bullet)\to \cal I^b(D_\bullet)$ is $\Sigma$-cartesian, where $\Sigma=(n{-}q_1{-}2){+}\dots {+}(n{-}q_r{-}2)$. 
\end{lem}

In proving this it will be convenient to work with equivalent statements about spaces of embeddings rather than spaces of interiors. View the $r$-cube $\cal I(D_\bullet)$ as a map of $(r-1)$-cubes 
$$
\cal I(D_{\bullet\cup \lbrace r\rbrace})\to \cal I(D_{\bullet}),
$$
a gluing-in-$Q_r$ map, where the subscript now runs through subsets of $\lbrace 1,\dots ,r-1\rbrace$. When a basepoint is chosen in $\cal I(D_{\underline{r-1}})$, giving compatible basepoints in all of the spaces $\cal I(D_{\bullet})$, then we can consider the cube of homotopy fibers and describe it as $E(Q_r,N-Q_\bullet)$. Here the basepoint has been interpreted as a manifold $N$ with boundary $D$ together with disjoint embeddings of $Q_1,\dots Q_{r-1}$ in it. Likewise the homotopy fiber of the block analogue
$$
\cal I^b(D_{\bullet\cup \lbrace r\rbrace})\to \cal I^b(D_{\bullet})
$$
can be described as $E^b(Q_r,N-Q_\bullet)$.\\

To obtain the conclusion of the lemma we show that (for every choice as above) the map
$$
E(Q_r,N-Q_\bullet)\to E^b(Q_r,N-Q_\bullet)
$$ 
is a $\Sigma$-cartesian $r$-cube. Renaming $(Q_1,\dots ,Q_{r-1},Q_r)$ as $(Q_1,\dots ,Q_r,P)$ (and renaming $r-1$ as $r$), this becomes:
$$
E(P,N-Q_\bullet)\to E^b(P,N-Q_\bullet).
$$
For a fixed choice of embedding of $P$ in $N$ disjoint from all $Q_i$ write $E^{\rel}(P,N)$ for the homotopy fiber of $E(P,N)\to E^b(P,N)$.
 
\begin{lem}[Restatement of \ref{conclemma}]\label{newconclemma}
 Let $N$ be a smooth compact $n$-manifold. Suppose that $r\ge 0$ and that $P,Q_1,\dots ,Q_r$ are disjoint codimension zero submanifolds with handle dimensions $p,q_1,\dots ,q_r$ all less than or equal to $n{-}3$. Then the $r$-cube 
$$
E^{\rel}(P,N-Q_\bullet)
$$ 
is $(n{-}p{-}2+\Sigma)$-cartesian where as usual $\Sigma=(n{-}q_1{-}2)+\dots {+}(n{-}q_r{-}2)$.
\end{lem}

We will deduce this from a statement about concordance embedding spaces, namely the main result of \cite{thesis} in the form of Lemma \ref{concordance-disjunction} below.

\subsection{Concordance Embedding Spaces}
Let $P$ be a submanifold of $N$.

\begin{defn} A {\it concordance embedding} is a smooth embedding $P \times I\to N \times I$
that restricts to the inclusion on $P\times 0 \cup \partial P \times I$ and takes $P\times 1$ into $N\times 1$. The {\it concordance embedding space} $\CE(P,N)$ is the simplicial set in which 
a $k$-simplex is a family of such embeddings smoothly parametrized by $\Delta^k$.
\end{defn}

Restriction to $P\times 1$ gives a fibration $\CE(P,N)\to E(P,N)$, and the fiber over the point corresponding to the inclusion $P\to N$ is \mbox{$E(P\times I,N\times I)$.} We can define a block version $\CE^b(P,N)$, giving a diagram
$$
\xymatrix{
E(P{\times}I,N{\times}I) \ar[r]\ar[d] & \CE(P,N) \ar[r]\ar[d] &   E(P,N) \ar[d]   \\
E^b(P{\times}I,N{\times}I) \ar[r] & \CE^b(P,N) \ar[r]& E^b(P,N)\, , 
}
$$
in which the rows are fibration sequences. The space $\CE^b(P,N)$ is contractible as long as the handle codimension is three or more. (This follows easily from the analogous statement with codimension instead of handle codimension, which is proved in \cite[lem.\ 2.1]{BLR}.) Thus the fiber $\CE^{\rel}(P,N)$ of the middle
vertical map is equivalent to $\CE(P,N)$.

If $P$ and $Q_1,\dots ,Q_r$ are disjointly embedded in $N$ then we have an $r$-cube $\CE(P,N-Q_\bullet)$. The following is essentially the main result of \cite{thesis}.

\begin{lem} \label{concordance-disjunction} If the handle dimensions $p$ and $q_i$ are all at most $n-3$, then the $r$-cube $\CE(P,N-Q_\bullet)$ is 
$(n{-}p{-}2{+}\Sigma)$-cartesian. Here $r$ can be any nonnegative integer. 
\end{lem}

(The statement in \cite{thesis} concerns the special case where handle dimension is dimension. We omit the argument for reducing the general case to that case, since it is exactly like the corresponding argument for $E(P,N-Q_\bullet)$ as explained in \ref{ThmA}.)

The case $r=1$ is Morlet's Disjunction Lemma. The case $r=0$ says that the space $\CE(P,N)$ is $(n{-}p{-}3)$-connected, a $k$-cartesian $0$-cube being the same thing as a $(k{-}1)$-connected space. In particular this recovers Hudson's result, that $\CE(P,N)$ is connected if $n{-}p\ge 3$.

\begin{proof}[Proof of \ref{newconclemma}]We begin with the case $r=0$. In the fibration sequence
$$
E^{\rel}(P\times I,N\times I)\to \CE^{\rel}(P,N)\to E^{\rel}(P,N).
$$
the middle space is $(n{-}p{-}3)$-connected, therefore $0$-connected. The second map is clearly $0$-connected, and therefore $E^{\rel}(P,N)$ is $0$-connected. We prove that it is $(n{-}p{-}3)$-connected by inductively proving that it is $k$-connected for $0\le k\le n{-}p{-}3$. By inductive hypothesis the left space is $(k{-}1)$-connected. As long as $k\le n{-}p{-}3$ the middle space is $k$-connected. Now $E^{\rel}(P,N)$ is $k$-connected because it is $0$-connected and its loopspace (the fiber of a map from the $(k{-}1)$-connected space $E^{\rel}(P\times I,N\times I)$ to the $k$-connected space $\CE^{\rel}(P,N)$) is $(k{-}1)$-connected.

The proof for $r>0$ is similar. Induct on $r$. Consider the diagram of $r$-cubes
$$
E^{\rel}(P{\times}I,N{\times}I-(Q{\times} I)_\bullet)\to \CE^{\rel}(P,N-Q_\bullet)\to E^{\rel}(P,N-Q_\bullet).
$$
We know that the spaces in the cube $E^{\rel}(P,N-Q_\bullet)$ are $0$-connected. We show by induction on $k$ that this cube is $k$-cartesian for $0\le k\le n-p-2+\Sigma$. To see that it is $0$-cartesian, view it as a map of $(r{-}1)$-cubes which by induction on $r$ are known to be $1$-cartesian. To go from $k{-}1$ to $k$, note that the cube $\Omega E^{\rel}(P,N-Q_\bullet)$ is $(k{-}1)$-cartesian, being the fiber of a map from a $(k{-}1)$-cartesian cube to a $k$-cartesian cube. A $0$-cartesian cube of based spaces must be $k$-cartesian if it becomes $(k{-}1)$-cartesian after looping.

\end{proof}

\section{End of the Main Proof \label{summing-up}}

Here we complete the proof of the main results in the case when all handle codimensions are at least three. 

We know (Lemma \ref{handlesplit}) that Theorem \ref{E} implies Theorem \ref{EF}. We also know, by Remark \ref{leeway}, that when $n{-}p\ge 3$ then 
Theorem \ref{E} follows from a slightly weakened form of Theorem \ref{EF} in which the connectivity is $n{-}2p{-}2{+}\Sigma$ rather 
than $n{-}2p{-}1{+}\Sigma$. By the proof of Lemma \ref{handlesplit}, this weakened Theorem \ref{EF} in turn follows from a similarly weakened Theorem  
\ref{E}, in which the connectivity is  ${-}p{+}\Sigma$ rather than $1{-}p{+}\Sigma$. Thus, in order to prove both of the main results in all cases where the handle codimensions are all at least three, it is enough to prove Theorem \ref{E} with the weakened conclusion: $E(P,N-Q_\bullet)$ is $({-}p{+}\Sigma)$-cartesian. 

Together, the results of the last few sections give us exactly that as long as $n\ge 5$. When $n{=}4$ they give that the cube is \emph{almost} $(-p{+}\Sigma)$-cartesian.

To finish off the low-dimensional cases we will use the result mentioned in Remark \ref{vw}, which gives the number $1{-}rp{+}\Sigma$ rather than $1{-}p{+}\Sigma$ under the hypotheses of Theorem \ref{E}. 

In the rather trivial case when $p{=}0$, these two numbers are equal, so that Theorem \ref{E} holds. Using the equivalence between Theorem \ref{E} and the more symmetrical Theorem \ref{symm}, we see that it also holds if $q_i{=}0$ for some $i$. Thus we may assume that $p>0$ and $q_i>0$.

Because we are assuming $n{-}p\ge 3$, this means that the only remaining case to consider is $n{=}4$, $p{=}1$, $q_i{=}1$. In this case $
-p{+}\Sigma_{i=1}^r(n{-}q_1{-}2)=r{-}1,
$ so the desired statement is that $E(P,N-Q_\bullet)$ is $(r{-}1)$-cartesian.
We know that it is almost $(r{-}1)$-cartesian (Addendum \ref{4}), so we have only to see that it is $0$-cartesian. In fact it is $1$-cartesian, since 
$$
1{-}rp{+}\Sigma{=}1.
$$

\section{Extension to Handle Codimension $\le 2$ \label{cod2}}

In this section we complete the proof of the two main results. That is, we eliminate the hypothesis that all handle codimensions are at least three and obtain the conclusions in all cases except that of classical knot theory.

First consider the case when the handle codimension of $P$ is at least three. Let $N$, $P$, and $Q_1,\dots ,Q_r$ be as in Theorem \ref{E} or Theorem \ref{EF}.

\begin{lem}\label{some3}If $n{-}p\ge 3$, then $E(P,N-Q_\bullet)$ is $(1{-}p{+}\Sigma)$-cartesian and $\EF(P,N-Q_\bullet)$ is $(n{-}2p{-}1{+}\Sigma)$-cartesian.
\end{lem}
\begin{proof}Let $j$ be the number of values of $i$ such that $n{-}q_i\le 2$, and argue by induction on $j$. The $j={0}$ case has been proved. Let $j$ be positive. Without loss of generality $n-q_r\le 2$.

The statement for $\EF(P,N-Q_\bullet)$ follows from the statement for $E(P,N-Q_\bullet)$ as in the proof of Lemma \ref{handlesplit}. (See also Remark \ref{cases}.)

To obtain the statement for $E(P,N-Q_\bullet)$, recall (Remark \ref{F}) that $F(P,N-Q_\bullet)$ is $(1{-}p{+}\Sigma)$-cartesian, so that 
it will suffice to show that
\mbox{$\EF(P,N-Q_\bullet)$} is $(1{-}p{+}\Sigma)$-cartesian. 
To see that it is, write it as a map of cubes
$$
\EF(P,N-(Q_\bullet\cup Q_r))\to \EF(P,N-Q_\bullet).
$$
where $\bullet$ now runs through subsets of $\underline {r-1}$. By induction on $j$ each of the two cubes is $(n{-}2p{-}1+\Sigma_{i=1}^{r-1}(n{-q_i}-2))$-cartesian, and therefore the map is $(n{-}2p{-}2+\Sigma_{i=1}^{r-1}(n{-}q_i{-}2))$-cartesian. This is greater than or equal to $1{-}p+\Sigma_{i=1}^{r}(n{-}q_i{-}2)$, because $n{-}2p{-}2\ge 1{-}p$ and ${n}-q_r-2\le 0$.
\end{proof}

We now prove the remaining cases of Theorem \ref{E}. Use the symmetrical version Theorem \ref{symm}, so that the desired statement is 
that $E(Q_\bullet,N)$ is $(3-n+\Sigma)$-cartesian.
By Lemma \ref{some3}, we have it in all cases in which $n-q_i\ge 3$ for some $i$. Thus we may assume $n-q_i\le 2$ for all $i$.

Cases where some $q_i$ is zero are covered by Remark \ref{vw}. Thus we may also assume $q_i\ge 1$ for all $i$.  Then

\begin{itemize}
\item $n$ cannot be $0$.
\item If $n{=}1$ then $q_i{=}1$ and $3{-}n{+}\Sigma= 2{-}2r<0$. 

\item If $n{=}2$ then $q_i\ge 1$ and $3{-}n{+}\Sigma\le 1{-}r<0$. 

\item If $n\ge 3$ then $3{-}n{+}\Sigma\le 0$, with equality only in the exceptional case when $n{=}3$ and $q_i{=}1$ for all $i$. 
\end{itemize}

Thus outside the exceptional case $n{=}3$ and $q_i{=}1$
there is nothing to prove.
This completes the proof of Theorem \ref{E}. Theorem \ref{EF} follows.

\section{Appendix I: Waldhausen's Theorems A' and B'}\label{WaldAB}

In \cite{Waldhausen_manifold}, Waldhausen gives variants  of Quillen's Theorems A and B in the context of simplicial categories. The purpose of this appendix is to state these results in the case that we need.

Suppose $f\: \cal A \to \cal B$ is a functor of simplicial categories.
We will assume $\text{ob}\,\cal A_k= \text{ob}\,\cal A_0$ and
$\text{ob}\,\cal B_k= \text{ob}\,\cal B_0$ for $k\in \Bbb N$.  For $b\in B_0$ an object, let
$$
f/b
$$
be the simplicial category which in simplicial degree $k$ is the left fiber 
$f_k/b$, where $f_k \: \cal A_k \to \cal B_k$ is the functor given by
restricting $f$ to simplicial degree $k$. (Note that our $f/b$ is the same as Waldhausen's
$f/([0],b)$.)  A morphism $b\to b'$ of $\cal B_0$ induces a simplicial
functor $f/b \to f/b'$ called a transition map.

\begin{thm}[Theorem A'] \label{A'}In addition to the above  assume that 
$f/b$ is contractible. Then the simplicial functor $f\: \cal A \to \cal B$
is a weak equivalence.
\end{thm} 

\begin{thm}[Theorem B']\label{B'} In addition to the above assume that each
transition map $b\to b'$ of $\cal B_0$ induces a weak equivalence
\mbox{$f/b \to f/b'$.} Then for every object $b\in \cal B_0$ the square
$$
\xymatrix{
f/b \ar[r] \ar[d] & \cal A \ar[d]^{f} \\
\text{\rm id}_{\cal B}/b \ar[r] & \cal B 
}
$$is homotopy cartesian.
\end{thm} 

Each of these statements is a special case of
the result of \cite[p.166]{Waldhausen_manifold} that appears in the addendum
on that page.

\section{Appendix II: Simple Poincar\'e Spaces \label{simple}}

Let $X$ be a connected space. A {\it finiteness structure} on $X$ consists of a finite complex $K$ together with
a choice of homotopy equivalence $(K,h\: K\to X)$. Say that two such structures $(K,h)$ and $K',h')$ are equivalent if the resulting homotopy equivalence $K\to K'$ is simple.
 
When a finiteness structure exists then the Whitehead group 
$\Wh(X)$ ($=\Wh(\pi)$, where $\pi = \pi_1(X)$) 
acts freely and transitively on the set of equivalence classes of finiteness structures. The rule is that an element $x$ takes $(K,h)$ to $(K',h\circ f)$, where the finite complex $K'$ and homotopy equivalence $f:K'\to K$ are chosen such that the torsion of $f$ corresponds to $\tau$ by $h$. 

There is a straightforward generalization to the following relative case: Fix a finite complex $D$, and say that a finiteness structure on a pair $(X,D)$ is a homotopy equivalence $h:(K,D)\to (X,D)$ where $(K,D)$ is a finite CW pair and $D\to D$ is the identity. Equivalence classes are defined using homotopy equivalences fixed on $D$, and again if the set of classes is nonempty then it has a canonical free and transitive action of $\Wh(X)$.

Now let $K$ be a finite complex satisfying $n$-dimensional Poincar\'e duality, with orientation bundle $\cal L$ and fundamental class $[K]\in H_n(K;\cal L)$. Let $\pi$ be the fundamental group and write $\Lambda = \Bbb Z[\pi]$. The duality map, cap product with $[K]$, leads to a chain map $C^*(K;\Lambda) \to C_{d-*}(K; \Lambda \otimes \cal L)$ from cellular cochains to cellular chains. This is a chain homotopy equivalence between free finite complexes of $\Bbb Z[\pi]$-modules, well defined up to chain homotopy. Denote its Whitehead torsion by $\tau_K$, and call $K$ simple if $\tau_K=0$. 

The torsion of a finiteness structure $(K,h)$ on a Poincar\'e complex $X$ is defined by $\tau_{K,h}=h_\ast\tau_K$. A simple structure on $X$ is a finiteness structure with zero torsion.

Let $\tau\mapsto \tau^*$ be the involution of the Whitehead group determined by the anti-involution $g\mapsto \epsilon(g)g^{-1}$ of the group ring $\Bbb Z[\pi]$, where $\epsilon:\pi\to \lbrace +1,-1\rbrace$ is the orientation character. The torsion of a finite Poincar\'e complex satisfies $\tau_K  = (-1)^n \tau^*_{K}$. Let $\frak N\:\Wh(X)\to \Wh(X)$ be the norm map $x\mapsto x{+}(-1)^nx^\ast$.
  
When two finiteness structures $(K,h)$ and $(K',h')$ on the Poincar\'e complex $X$ are related by a homotopy equivalence $f:K'\to K$ (which may be taken to be cellular), then there is a homotopy commutative diagram of chain equivalences 
$$
\xymatrix{
C^*(K';\Lambda) \ar[r]^{f^*} \ar[d]_{\cap [K']} & C^*(K;\Lambda) \ar[d]^{\cap [K]}\\
C_{d-*}(K';\Lambda^t) & C_{d-*}(K;\Lambda^t)\ar[l]^{f_*}
}
$$
The composition formula for Whitehead torsion 
\cite[22.4]{Cohen} gives that
 $\tau_{K',h'} {-} \tau_{K,h} = x {+} (-1)^n x^*$, where $x \in \Wh(X)$
is the torsion of $f$. We conclude the following: First, if $X$ has a finiteness structure then the torsion classes of all such structures belong to one element of the cokernel of $\frak N$. Second, a simple structure exists if and only if this element of the cokernel is zero. Third, in this case the free transitive action of $\Wh(X)$ on the set of finiteness structures restricts to give a free transitive action of the kernel of $\frak N$ on the set of simple structures. 

Again there is a straightforward relative version. Suppose that the pair $(X,D)$ satisfies $n$-dimensional Poincar\'e duality and that the $(n{-}1)$-dimensional Poincar\'e complex $D$ is simple. Then for an equivalent finite $(K,D)$ the torsion may be defined by the cap product \mbox{$C^*(K;\Lambda) \to C_{d-*}(K,D; \Lambda \otimes \cal L)$} or by the cap product $C^*(K,D;\Lambda) \to C_{d-*}(K; \Lambda \otimes \cal L)$ (it is the same because $D$ is simple), and the rest of the story is as in the absolute case.

\end{document}